\documentclass[a4paper,11pt,reqno]{amsart}

\usepackage{amsmath,amssymb,amsthm,amsaddr}
\usepackage[margin=3.5cm,top=3.5cm,bottom=3.5cm,footskip=1cm,headsep=1.5cm]{geometry}
\usepackage[utf8]{inputenc}
\usepackage[english]{babel}
\usepackage[font=footnotesize,labelfont=bf]{caption} 
\usepackage{graphicx}
\usepackage{booktabs}
\usepackage{hyperref}

\def\RR{\mathbb{R}}
\def\div{\operatorname{div}}
\def\curl{\operatorname{curl}}
\def\dom{\operatorname{dom}}
\def\epi{\operatorname{epi}}
\def\prox{\operatorname{prox}}
\def\argmin{\operatorname{arg\,min}}
\def\bb{\mathbf{B}}
\def\hh{\mathbf{H}}
\def\jj{\mathbf{J}}
\def\js{\mathbf{j_s}}
\def\hs{\mathbf{H_s}}
\def\ncp{\varphi_{\text{NCP}}}
\def\aa{\mathbf{A}}
\def\uu{\mathbf{u}}
\def\vv{\mathbf{v}}
\def\nn{\mathbf{n}}
\def\Th{\mathcal{T}_h}

\newtheorem{lemma}{Lemma}
\newtheorem{theorem}[lemma]{Theorem}

\theoremstyle{definition}
\newtheorem{assumption}[lemma]{Assumption}
\newtheorem{problem}[lemma]{Problem}
\newtheorem{remark}[lemma]{Remark}

\begin{document}
\title[A semi-smooth Newton method for magnetic hysteresis]{A semi-smooth Newton method for\\ magnetic field problems with hysteresis}
\author{H. Egger$^{1,2}$ \and  F. Engertsberger$^1$}
\address{%
\small 
$^1$Institute of Numerical Mathematics, Johannes Kepler University Linz, Austria \\
$^2$Johann Radon Institute for Computational and Applied Mathematics, Linz, Austria}

\begin{abstract}
Ferromagnetic materials exhibit anisotropy, saturation, and hysteresis. We here study the incorporation of an incremental vector hysteresis model representing such complex behavior into nonlinear magnetic field problems both, from a theoretical and a numerical point of view.  We show that the hysteresis operators, relating magnetic fields and fluxes at every material point, are strongly monotone and Lipschitz continuous. This allows to ensure well-posedness of the corresponding magnetic field problems and appropriate finite element discretizations thereof. We further show that the hysteresis operators are semi-smooth, derive a candidate for their generalized Jacobians, and establish global linear and local superlinear convergence of a the semi-smooth Newton method with line search applied to the iterative solution of the discretized nonlinear field problems. The results are proven in detail for a hysteresis model involving a single pinning force and the scalar potential formulation of magnetostatics. The extension to multiple pinning forces and the vector potential formulation is possible and briefly outlined. The theoretical results are further illustrated by numerical tests.
\end{abstract}

\maketitle 

\begin{quote}
\end{quote}

\begin{quote}
\end{quote}

\section{Introduction}
\label{sec:intro}
The simulation, optimization, and design of electric machines and power transformers requires the repeated numerical solution of magnetic field problems including nonlinear constitutive relations that accurately describe anisotropy, saturation, and hysteresis. 
The magnetostatic approximation of Maxwell's equations is suitable in the low frequency setting of such applications~\cite{Meunier2008,Salon2012}. 
As the basic model problem for our further consideration, we therefore consider the system
\begin{alignat}{3}
\curl \hh &= \mathbf{j_s} \quad \text{in } \Omega, \qquad & \div \bb &= 0 \quad &&\text{in } \Omega, \label{eq:1}\\
\bb \cdot \mathbf{n} &= 0 \quad \text{on } \partial\Omega, \qquad & \bb &= \mu_0 \hh + \jj \quad &&\text{in } \Omega. \label{eq:2}
\end{alignat}
Here $\hh$ denotes the magnetic field intensity, $\bb$ the magnetic induction, $\mu_0$ the permeability of vacuum, and $\jj$ the magnetic polarization density. From the first equation it is clear, that the current density $\js = \curl \hs$ can be represented by an appropriate source field $\hs$. The system is closed by the vector hysteresis model~\cite{Lavet2013}
\begin{align} \label{eq:3}
    \jj = \argmin\nolimits_{\jj}  \{ U(\jj) - \langle \hh ,\jj \rangle + \chi |\jj - \jj_{p}| \},
\end{align}
which describes the local change of the polarization $\jj$ from its previous state $\jj_p$ due to incremental loading by the magnetic field $\hh$. Here $U(\jj)$ is an internal energy density and $\chi \ge 0$ a parameter related to the strength of the pinning forces. 
The value $\jj_p$ serves as an internal variable encoding the memory of the system. Repeated solution of the system \eqref{eq:1}--\eqref{eq:3}, together with the update $\jj \to \jj_p$ after every load step, allows the simulation of load cycles of relevance in applications.

\subsection{Energy-based hysteresis model}
The model \eqref{eq:3} was proposed by Henrotte et al. in \cite{Lavet2013,Henrotte2006}, based on previous work of Bergqvist~\cite{Bergqvist1997,Bergqvist1997a}, and further analysed in \cite{Kaltenbacher2022,Prigozhin2016}. 
In contrast to some other popular models for hysteresis, like Preisach operators \cite{Mayergoyz1988} or the Jiles-Atherton model \cite{Jiles1986,Sadowski2002}, the above model 
is thermodynamically consistent and vectorial by construction. 
Let us briefly comment on the mathematical background. 
The first order optimality condition for \eqref{eq:3} reads 
\begin{align}
\label{eq:inclusion}
\nabla_{\jj} U(\jj) - \hh \in 
\begin{cases}
-\chi \frac{\jj - \jj_{p}}{|\jj - \jj_{p}|} & \text{if } \jj \neq \jj_{p}, \\
\{x \in \RR^d : |x| \leq \chi\} &  \text{if } \jj = \jj_{p}.
\end{cases}
\end{align}
We note that the right hand side of this relation, which amounts to the sub-gradient of the function $\chi |\jj-\jj_p|$, depends on the direction but not on the size of $\jj-\jj_p$. 
Hence \eqref{eq:inclusion} can be understood as an implicit Euler discretization of the inclusion
\begin{align}
\hh - \nabla_{\jj} U(\jj) \in \partial_{\dot \jj} \mathcal{R}(\dot \jj) 
\end{align}
with dissipation potential $\mathcal{R}(\dot \jj) = \chi |\dot \jj|$. Together with \eqref{eq:1}--\eqref{eq:2}, this model falls into the class of rate-independent systems studied in \cite{Mielke2015}.
The problem \eqref{eq:1}--\eqref{eq:3} hence amounts to a single load step of of such a system after discretization in time.

\subsection{Magnetic field problem with hysteresis}
The local material model \eqref{eq:3} defines $\jj=\jj(\hh;\jj_p)$ as a function of the magnetic field $\hh$ and the polarization $\jj_p$. The second equation in \eqref{eq:2} with \eqref{eq:3} thus implicitly defines a constitutive law 
\begin{align} \label{eq:scal0}
\bb(\hh;\jj_p) = \mu_0 \hh + \jj(\hh;\jj_p),
\end{align}
called \emph{forward hysteresis operator}, which describes the relation between magnetic fields and fluxes. 
On topologically simple domains, the magnetic field can be expressed as $\hh=\hh_s - \nabla \psi$ with $\psi$ denoting a magnetic scalar potential~\cite{Meunier2008,Monk2003}. Then the first equation in \eqref{eq:1} is satisfied automatically and the system \eqref{eq:1}--\eqref{eq:3} can be restated equivalently as a second order elliptic boundary value problem
\begin{align}
\div (\bb(\hh_s - \nabla \psi;\jj_p)) &= 0 \qquad \text{in } \Omega, \label{eq:scal1}\\
\bb(\hh_s - \nabla \psi;\jj_p) \cdot \nn &= 0 \qquad \text{on } \partial\Omega, \label{eq:scal2}
\end{align}
which characterizes the magnetic scalar potential $\psi$ for a single time step. The simulation of a full load cycle then requires to solve this system for varying excitations $\js = \curl \hh_s$ and incrementally updating $\jj \to \jj_p$ after each load step.

\subsection{Related work and main contributions}
The efficient evaluation of the hysteresis operator $\bb(\hh;\jj_p)$ has been investigated intensively in the literature and its application for simulations based on the scalar potential formulation \eqref{eq:scal1}--\eqref{eq:scal2} has been demonstrated successfully~\cite{Prigozhin2016,Jacques2018}. 
Various methods have been proposed for the iterative solution of the nonlinear systems arising from discretization of the nonlinear magnetic field problems; see \cite{Lavet2013,Egger2024quasi,Jacques2015} for some results.  
In this work, we make the following contributions to this active field of research:
\begin{itemize}
\item We show that the mapping $\bb(\hh;\jj_p) = \nabla_\hh w^*(\hh;\jj_p)$ is well-defined, Lipschitz continuous, and strongly monotone, and can be expressed as the derivative of a convex co-energy density $w^*(\hh;\jj_p)$; see Section~\ref{sec:material}. 
\item As a direct consequence, we obtain existence of a unique solution to \eqref{eq:scal1}--\eqref{eq:scal2} and appropriate finite element discretizations thereof; see Section~\ref{sec:scalar}.
\end{itemize}
Our analysis generalizes previous results \cite{Engertsberger23,Egger2024global} for smooth material laws to the case of hysteresis, for which $\bb(\hh;\jj_p)$ is in general not differentiable; see below.
\begin{itemize}
\item We further show that the mapping $\bb(\hh;\jj_p)$ is semi-smooth~\cite{Qi1993,Ulbrich2011}, i.e., it can be well-approximated by certain linearizations, and we provide an explicit formula for a representative in the generalized Jacobian; see Section~\ref{sec:semismooth}.
\item We formulate a semi-smooth Newton method for the iterative solution of the discretized nonlinear problem, prove its global linear and mesh-independent convergence, as well as local superlinear convergence; see Section~\ref{sec:newton}.
\end{itemize}
These latter results are strongly motivated by corresponding analyses for elasto-plasticity~\cite{Gruber2009,Hager2009,Wieners2011} which share many common features with the magnetic hysteresis models considered here; see \cite{Mielke2015} for a discussion of the common basis.
The results are first derived for the hysteresis model \eqref{eq:3} with one pinning force. The extension to multiple pinning forces is possible and will be briefly discussed in Section~\ref{sec:extension}. 
Numerical results are presented in Section~\ref{sec:numerics} to illustrate the efficiency of the semi-smooth Newton method and to give a comparison to other iterative solvers proposed in previous work. The theoretical results provided in this paper, in fact, also provide a solid theoretical backup for some of these methods. 
\section{Local material law}
\label{sec:material}

In this section, we prove some elementary facts about the constitutive relations of the energy-based hysteresis model which are at the core for the further investigations.

\subsection{Preliminaries}
Let us recall some basic notions of convex analysis~\cite{Bauschke2011,Beck2017}: A function $f : \RR^d \to (-\infty,\infty]$ to the extended reals is called \emph{closed}, if its epigraph $\epi(f)=\{(x,y) : y \ge f(x)\}$ is closed. It is called \emph{proper}, if $f$ is not unbounded everywhere, and $\dom(f) = \{x : f(x) < \infty\}$ denotes the domain of $f$. Furthermore, $f$ is called \emph{$\sigma$-strongly convex}, if the function $g(x) = f(x) - \frac{\sigma}{2} |x|^2$ is still convex. 
Now let $f$ be convex, closed, and proper. Then for any $\mu>0$ and $x \in \RR^d$, the minimizer
\begin{align}
\prox_{\mu f}(x) = \arg\min_y \{f(y) + \tfrac{1}{2\mu} |y-x|^2\}
\end{align}
is well-defined and $x \mapsto \prox_{\mu f}(x)$ is called \emph{proximal operator}. 
The minimal value 
\begin{align} 
M_f^\mu(x) = \min_y \{f(y) + \tfrac{1}{2\mu} |y-x|^2\}
\end{align}
is called the \emph{Moreau envelope}. The following properties are well-known. 
\begin{lemma} \label{lem:prelim}
Let $f : \RR^d \to (-\infty,\infty]$ be closed, convex, and proper, and $\mu>0$. 
Then 

\noindent 
(i) The mapping $x \mapsto \prox_{\mu f}(x)$  is well-defined on $\RR^d$, monotone, and non-expansive, i.e., $\langle \prox_{\mu f}(x) - \prox_{\mu f}(y), x-y\rangle \ge 0$ and $|\prox_{\mu f}(x) - \prox_{\mu f}(y)| \le |x-y|$. 

\noindent 
(ii) The mapping $x \mapsto M_f^\mu(x)$ is convex and continuously differentiable on $\RR^d$ with Lipschitz continuous derivative $\nabla M_f^\mu(x) = \frac{1}{\mu} (x - \prox_{\mu f}(x))$. 
\end{lemma}
%
We refer to \cite[Ch.~6.6]{Beck2017} for a proof of these assertions and further related results. 

\subsection{Local material model}
For the analysis of the hysteresis model \eqref{eq:3}, we make the following assumptions, which actually cover quite general classes of materials.
\begin{assumption}
\label{ass:1}
$\jj_p \in \RR^d$, $d=2,3$, $\chi \ge 0$, and $U : \RR^d \to (-\infty,\infty]$ is proper, closed, $\sigma$-strongly convex, and two-times continuously differentiable on its domain. 
\end{assumption}
A typical example for $U$ satisfying these assumptions is $U(\jj) =-\frac{2 A J_s}{\pi} \log(\cos(\frac{\pi}{2}\frac{\jj}{J_s}))$ with positive constants $A$, $J_s$; see e.g.~\cite{Prigozhin2016}.
The following result establishes the properties of the hysteresis operator $\bb(\hh;\jj_p)$ mentioned in the introduction. For ease of notation, we do not mention the dependence on $\jj_p$ in the following, unless needed.
\begin{theorem} \label{thm:w}
Let Assumption~\ref{ass:1} be valid. Then the co-energy density
\begin{align} \label{eq:ws}
w^*(\hh) = \frac{\mu_0}{2} |\hh|^2 - \inf\nolimits_{\jj} \{ U(\jj)  - \langle \hh,\jj\rangle + \chi |\jj-\jj_p| \}
\end{align}
is well-defined and continuously differentiable on $\RR^d$ with
$\nabla_\hh w^*(\hh) = \mu_0 \hh + \jj(\hh)$ where $\jj(\hh)$ denotes the minimizer of \eqref{eq:3}. 
Moreover
\begin{align}
\langle \nabla_\hh w^*(\hh) - \nabla_\hh w^*(\tilde \hh), \hh - \tilde \hh\rangle &\ge \mu_0 |\hh - \tilde \hh|^2, \\
|\nabla_\hh w^*(\hh) - \nabla_\hh w^*(\tilde \hh)| &\le \tfrac{1+\sigma \mu_0}{\sigma} |\hh-\tilde \hh|
\end{align}
i.e., the gradient $\nabla_\hh w^*(\hh)$ is strongly monotone and Lipschitz continuous.
\end{theorem}
\begin{proof}
Let us define $g(\jj) = U(\jj) - \frac{\sigma}{2} |\jj|^2 + \chi |\jj - \jj_p|$. 
Then by elementary manipulations, the co-energy density \eqref{eq:ws} can be rewritten equivalently as
\begin{align*}
w^*(\hh) = \tfrac{\mu_0}{2} |\hh|^2 + \tfrac{1}{2\sigma} |\hh|^2 - \inf\nolimits_{\jj} \{g(\jj) + \tfrac{\sigma}{2} |\hh/\sigma - \jj|^2\}.
\end{align*}
The second term amounts to the Moreau envelope $M_g^{1/\sigma}(\hh/\sigma)$, and by Lemma~\ref{lem:prelim} point (ii) and the chain rule, we see that 
\begin{align*}
\nabla_\hh M_g^{1/\sigma}(\hh/\sigma) 
= \frac{1}{\sigma} \left[ \sigma \left( \hh/\sigma - \prox_{g/\sigma}(\hh/\sigma)\right)\right].
\end{align*}
By combination with the previous formula, we thus obtain 
\begin{align*}
\nabla_\hh w^*(\hh) = \mu_0 \hh + \prox_{g/\sigma}(\hh/\sigma). 
\end{align*}
From Lemma~\ref{lem:prelim} point (i), we conclude that $\nabla_\hh w^*(\hh)$ is strongly convex and Lipschitz continuous with parameters $\mu_0$ and $\mu_0+\frac{1}{\sigma} = \frac{\sigma \mu_0 + 1}{\sigma}$, which already yields the estimates of the theorem. 
From definition of the prox-operator, we further see that 
\begin{align*}
\prox_{g/\sigma}(\hh/\sigma)
&= \arg\min_\jj \{g(\jj) + \tfrac{\sigma}{2} |\hh/\sigma - \jj|^2\} \\
&= \arg\min_\jj \{U(\jj) - \langle \hh,\jj\rangle + \chi |\jj-\jj_p| \} = \jj(\hh),
\end{align*}
which again follows by elementary computations. This shows the formula for the gradient $\nabla_\hh w^*(\hh)$ and already concludes the proof of the assertions.
\end{proof}

\begin{remark}
The previous theorem shows that the energy-based vector hysteresis operator of \cite{Lavet2013} can be expressed as $\bb(\hh) = \nabla_\hh w^*(\hh)$ and it defines a strongly monotone and Lipschitz continuous mapping between $\hh$ and $\bb$. 
\end{remark}

\section{The scalar potential formulation and its discretization}
\label{sec:scalar}
Based on the previous considerations, we can now establish the well-posedness of the scalar potential formulation \eqref{eq:scal1}--\eqref{eq:scal2} and appropriate finite element approximations thereof. These results follow by standard arguments, see e.g. \cite{Engertsberger23,Zeidler1990}, but for later reference, we state them in detail and provide short proofs.

\subsection{Scalar potential formulation}
For the subsequent analysis, we require the following conditions, which cover rather general cases arising in applications.
\begin{assumption} \label{ass:2}
$\Omega \subset \RR^d$, $d=2,3$ is a bounded Lipschitz domain which is 
homeomorphic to a ball. 
Further $0 \le \chi \in L^\infty(\Omega)$ and 
$U : \Omega \times \RR^d \to (-\infty,\infty]$ is measureable and for a.a. $x \in \Omega$, the function $U(x,\cdot) : \RR^d \to (-\infty,\infty]$ is $\sigma$-strongly convex and two times continuously differentiable with $\sigma>0$ independent of $x$.
Finally, $\js = \curl \hh_s$ for some $\hh_s \in H(\curl;\Omega) = \{u \in L^2(\Omega)^d : \curl u \in L^2(\Omega)^{d/(d-1)}\}$.
\end{assumption}
For details on the notation and the typical function spaces arising in the context of Maxwell's equations, we refer to \cite{Monk2003,Boffi2013}.
Since the topology of the domain $\Omega$ is assumed trivial, we may decompose any $\hh \in H(\curl;\Omega)$ into $\hh = \hh_s - \nabla \psi$ with a scalar potential $\psi \in H^1(\Omega)$, and verify that \eqref{eq:1}--\eqref{eq:3} is equivalent to \eqref{eq:scal0}--\eqref{eq:scal2}. 
It is not difficult to see that, if $\psi$ is a solution of these equations, then also $\psi + c$ with $c \in \RR$ is a solution. To guarantee uniqueness, we require that $\int_\Omega \psi \, dx = 0$ in the sequel. 
The weak form for this problem thus reads as follows. 
\begin{problem} \label{prob:scal}
Find $\psi \in H^1_*(\Omega) = \{v \in H^1(\Omega) : \int_\Omega v \, dx = 0\}$ such that 
\begin{align} \label{eq:scal_var}
\int_\Omega \nabla_\hh w^*(\hh_s - \nabla \psi) \cdot \nabla v \, dx = 0 \qquad \forall v \in H_*^1(\Omega).
\end{align}
\end{problem}
The variational identity \eqref{eq:scal_var} is derived in the usual manner, by multiplying \eqref{eq:scal1} with a test function $v$, integration over the domain $\Omega$, integration-by-parts, and noting that due to \eqref{eq:scal2} the boundary terms vanish. 
The dependence of $w^*$ on the spatial point $x$ was dropped in our notation.
By using the results of the previous section, we immediately obtain the following assertion, already announced in \cite{Egger2024quasi}.
\begin{theorem}
Let Assumption~\ref{ass:2} hold. 
Then Problem~\ref{prob:scal} admits a unique solution which is characterized equivalently as the unique minimizer of
\begin{align} \label{eq:scal_min}
\min_{v \in H^1_*(\Omega)} \int_\Omega w^*(\hh_s - \nabla v) \, dx.
\end{align}
\end{theorem}
\begin{proof}
By Assumption~\ref{ass:2} and Theorem~\ref{thm:w},
the co-energy density $w^*(x,\cdot)$ is continuously differentiable and $\nabla_\hh w^*(x,\cdot)$ is strongly monotone and Lipschitz continuous for all $x \in \Omega$. 
Hence by Zarantonello's Lemma~\cite[Theorem 25.B]{Zeidler1990}, the variational problem \eqref{eq:scal_var} admits a unique solution.
Moreover, \eqref{eq:scal_var} are the first order optimality conditions for \eqref{eq:scal_min}, which by convexity characterize the minimizer.
\end{proof}

\subsection{Finite element discretization}
We now study the numerical solution of the variational problem \eqref{eq:scal_var} by appropriate finite element methods. For ease of presentation, we only consider the lowest order approximation on simplicial grids.

Thus let $\Th = \{T\}$ be a geometrically conforming partition of $\Omega \subset \RR^d$ into simplices, i.e., triangles ($d=2$) or tetrahedra ($d=3$). 
We denote by 
\begin{align}
V_h = \{v_h \in H^1_*(\Omega): v_h |_{T} \in P_1(T) \ \forall T \in \Th\}
\end{align}
the space of continuous and piecewise linear finite element functions~\cite{Boffi2013,Braess2007}, 
which additionally have integral zero. 
In order to adequately handle the nonlinear terms in the implementation, 
we approximate the $L^2$-scalar product on $\Omega$ by 
\begin{align} \label{eq:int_h}
(\uu_h,\vv_h)_h = \sum\nolimits_{T \in \Th}  \uu_h(x_T) \cdot \vv_h(x_T) \, |T|,
\end{align}
where $x_T$ and $|T|$ denotes the barycenter and the area of the element $T$. 
The corresponding norm is denoted by $\|\uu_h\|_h = \sqrt{(\uu_h,\uu_h)_h}$.
As approximation for the scalar potential formulation \eqref{eq:scal_var}, 
we then consider the following system.
\begin{problem} \label{prob:scal_h}
Find $\psi_h \in V_h$ such that 
\begin{align} \label{eq:scal_var_h}
(\nabla_\hh w^*(\hh_s - \nabla \psi_h),v_h)_h &= 0 \qquad \forall v_h \in V_h.
\end{align}
\end{problem}
In order to make this formulation well-defined, we require some additional regularity of the problem data, which we formally introduce for later reference.
\begin{assumption} \label{ass:3}
Let Assumption~\ref{ass:2} hold and $\Th$ be a geometrically conforming simplicial partition of the domain $\Omega$. Further assume that $\chi$, $U$, and $\hh_s$ are piecewise continuous functions in space with respect to the partition $\Th$. 
\end{assumption}
In a similar manner as on the continuous level, we obtain the following result.
\begin{theorem} \label{thm:scal_h}
Let Assumption~\ref{ass:3} hold. Then Problem~\ref{prob:scal_h} admits a unique solution which is equivalently characterized as the unique minimizer of 
\begin{align} \label{eq:scal_min_h}
\min_{v_h \in V_h}  (w^*(\hh_s - \nabla v_h),1)_h.
\end{align}
\end{theorem}
\begin{proof}
Problem \eqref{eq:scal_var_h} amounts to a finite dimensional nonlinear system with strongly monotone and Lipschitz continuous operator. Existence of a unique solution thus follows by the Zarantonello Lemma \cite[Theorem 25.B]{Zeidler1990}. 
Condition \eqref{eq:scal_var_h} again amounts to the first order optimality conditions for the convex minimization problem \eqref{eq:scal_min_h}.
\end{proof}

\begin{remark}
The result of Theorem~\ref{thm:scal_h} can be generalized in various directions, e.g., domains with curved boundaries or higher order approximations, and the variational structure is also suitable for a discretization error analysis. Some related results for smooth anhysteretic material laws can be found in \cite{Engertsberger23,Egger2024ho}. 
\end{remark}

\section{Semi-smoothness of the material law}
\label{sec:semismooth}


Due to the non-smooth term in the material model \eqref{eq:3}, the corresponding hysteresis 
operator $\bb(\hh) = \nabla_\hh w^*(\hh)$ is not differentiable in a classical sense if $\chi > 0$. 
We will now show, however, that the local material model $\bb(\hh)$ is at least \emph{semi-smooth}, i.e., it can be approximated well by appropriate linearizations~\cite{Ulbrich2011,Qi1993b}. 
This will be used in the following section to propose a semi-smooth Newton method and establish its local and global convergence properties.

\subsection{Generalized derivatives and semi-smoothness}
Let us first recall some basic terms and notation of semi-smooth analysis; see~\cite{Ulbrich2011,clarke1990} for an introduction to the field. Let $F : W \subset \RR^n \to \RR^m$ be a locally Lipschitz continuous function. Then by Rademacher's Theorem the set
$D_F(x) := \{\hat{x} \in \RR^n : F \text{  is differentiable at } \hat{x}\}$ is dense in $W$. The \emph{Bouligand-subdifferential} is defined by the set-valued mapping
\begin{align}
    \partial^B F(x) := \{M \in \RR^{m \times n}: \exists (x_k) \in D_F : x_k \rightarrow x, \nabla F(x_k) \rightarrow M\},
\end{align}
which by construction is nonempty for every $x \in W$. \emph{Clarke's generalized Jacobian} of $F$ is defined as the convex hull $\partial F(x) = \text{co}(\partial^B F(x))$ of this set. Finally, the function $F$ is called  \emph{semi-smooth at $x$}, if it is locally Lipschitz continuous, directionally differentiable, and satisfies the approximation property 
\begin{align} \label{eq:smismooth}
\sup_{L \in \partial F(x + \delta x)} |F(x + \delta x) - F(x) - L \delta x| = o(|\delta x|) \quad \text{for} \quad \delta x \to 0.
\end{align}
It is called \emph{semi-smooth}, if it satisfies these properties for every point $x \in W$. 

\subsection{Semi-smoothness of the hysteresis operator} 
The remainder of this section will be devoted to proving the following key result of our paper. 
\begin{theorem}
    \label{thm:semismooth:b}
    Let Assumption \ref{ass:1} hold and $\jj=\jj(\hh)$ denote the minimizer of \eqref{eq:3}. Then the material law $\bb(\hh) = \nabla_{\hh} w^*(\hh)$ is semi-smooth and the matrix
    \begin{align}
    \label{eq:slant:b}
    \mathbb{S}_{\bb} = \left\{
    \begin{array}{ll}
    \mu_0 I, & \text{ if } \jj = \jj_p, \\
    \mu_0 I + \Big( \nabla_{\jj \jj}^2 U(\jj) + \frac{\chi}{|\jj - \jj_p|} \Big( I - \frac{(\jj - \jj_p) \,\otimes\, (\jj - \jj_p)}{|\jj - \jj_p|^2}\Big) \Big)^{-1}, & \text{ if } \jj \neq \jj_p, 
    \end{array}
    \right.
    \end{align}
    defines an element in the generalized Jacobian $\partial_{\hh} \bb(\hh)$, suitable for computations.
\end{theorem}
\begin{remark}
For $\chi=0$, the magnetic polarization $\jj(\hh)$ is characterized as minimizer of a smooth functional and by the implicit function theorem, one can see that $\nabla_\hh \jj(\hh) = \nabla_{\jj\jj}^2 U(\jj)^{-1}$. In this case $\mathbb{S}_{\bb} = \mu_0 I + \nabla_{\jj\jj}^2 U(\jj)^{-1} = \nabla_\jj \bb(\hh)$ amounts to the standard Jacobian of the material law. 
For $\chi>0$, the minimizer $\jj(\hh)$ is, however, not differentiable in the classical sense. The formula in the second line can however still be motivated by formal computations, as was done in \cite{Jacques2015}. 
Theorem~\ref{thm:semismooth:b} provides an extension and a  rigorous justification of these formal computations.
\end{remark}

For the proof of Theorem~\ref{thm:semismooth:b}, we utilize arguments employed in \cite{Hager2009,Wieners2011} for the analysis of material models in elasto-plasticity.
In a first step, we characterize $\jj(\hh)$ as the unique solution of a nonlinear complementarity problem, which is equivalent to a nonlinear system $T(\jj(\hh),\hh)=0$ involving a semi-smooth mapping $T$. An extension of the implicit function theorem then yields semi-smoothness of $\jj(\hh)$ and a suitable candidate in the generalized Jacobian $\partial_\hh \jj(\hh)$. The corresponding result for the hysteresis operator $\bb(\hh)=\mu_0 \hh + \jj(\hh))$ then follows by the additive structure.

\subsection{Equivalent characterizations of $\jj(\hh)$}
In a first step, we characterize the minimizer of \eqref{eq:3} as the unique solution to a nonlinear complementary system.
\begin{lemma}\label{lem:system}
Let Assumption \ref{ass:1} hold. Then $\jj \in \RR^d$ is the unique minimizer of \eqref{eq:3} if, and only if, there exists a unique $\lambda \in \RR$ such that $(\jj,\lambda)$ solves
\begin{align}
    \jj - \jj_p + \lambda (\nabla_{\jj} U(\jj) - \hh) &= 0 \label{eq:system:1}\\
    \lambda &\geq 0 \label{eq:system:2}\\
    |\nabla_{\jj} U(\jj) - \hh| - \chi &\leq 0 \label{eq:system:3}\\
    \lambda (|\nabla_{\jj} U(\jj) - \hh| - \chi) &= 0. \label{eq:system:4}
\end{align}
\end{lemma}
\begin{proof}
Let $\jj \in \RR^d$ be the unique solution of \eqref{eq:3}, Then $\jj$ is characterized by the first order necessary and sufficient optimality conditions
\begin{align}
    0 \in \nabla_{\jj} U(\jj) - \hh + \chi \partial_{\jj} |\jj - \jj_p|.
    \label{eq:foc:inclusion}
\end{align}
For the case $\jj = \jj_p$ the inclusion \eqref{eq:foc:inclusion} implies that $|\nabla_{\jj} U(\jj) - \hh| \leq \chi$ holds. Then $\lambda=0$ is the only choice to solve the system \eqref{eq:system:1}--\eqref{eq:system:4}. 
If $\jj \neq \jj_p$ the first order optimality condition \eqref{eq:foc:inclusion} has the form
\begin{align}
    0 = \nabla_{\jj} U(\jj) - \hh + \chi \frac{\jj - \jj_p}{|\jj - \jj_p|},
    \label{eq:foc:inclusion:diff}
\end{align}
which together with \eqref{eq:system:1} implies $\lambda = \frac{|\jj - \jj_p|}{\chi}$. From equation \eqref{eq:foc:inclusion:diff} we further conclude that $|\nabla_{\jj} U(\jj) - \hh| = \chi$, which also shows that \eqref{eq:system:3} and \eqref{eq:system:4} hold.

Now let ($\jj,\lambda$) be a solution of \eqref{eq:system:1}--\eqref{eq:system:4}. 
If $\jj = \jj_p$ we conclude from \eqref{eq:system:3} that the first order optimality condition \eqref{eq:foc:inclusion:diff} of problem \eqref{eq:3} is fulfilled, which is sufficient since the problem is strictly convex.
If $\jj \neq \jj_p$, then from \eqref{eq:system:1} and \eqref{eq:system:4}, we see that $\lambda \neq 0$ and  $|\nabla_{\jj} U(\jj) - \hh| = \chi$. 
Using this in \eqref{eq:system:1} gives $\lambda = \frac{|\jj - \jj_p|}{\chi}$, which yields the first order optimality condition for \eqref{eq:3} in the differentiable case.
\end{proof}

This alternative characterization of $\jj(\hh)$ will play a major role in the subsequent proofs. Using standard arguments~\cite{Ulbrich2011}, we first transform the complementarity problem into another equivalent nonlinear system.
\begin{lemma} \label{lem:T}
Let Assumption~\ref{ass:1} hold. Then the solution $(\jj,\lambda)$ of problem \eqref{eq:system:1}--\eqref{eq:system:4} can be equivalently characterized by the nonlinear equation
\begin{align} \label{eq:T} 
    T((\jj,\lambda);\hh) = \left(\begin{array}{c} \jj - \jj_p + \lambda (\nabla_{\jj} U(\jj) - \hh) \\ \ncp(\frac{1}{2}(|\nabla_{\jj}U(\jj) - \hh|^2 - \chi^2),\lambda) \end{array}  \right) = 0,
\end{align}
with complementarity function $\ncp(x_1,x_2) := \max(0,x_1+x_2) - x_2$. 
\end{lemma}
\begin{proof}
It is not difficult to see that 
\begin{align*}
    \ncp(x_1,x_2) = 0 \quad \Longleftrightarrow \quad x_1 \leq 0, \quad x_2 \geq 0, \quad x_1 x_2 = 0,
\end{align*}
which immediately provides the claim of the lemma.
\end{proof}

\subsection{Semi-smoothness of the minimizer $\jj(\hh)$}
We will utilize the following extension of the implicit function theorem for semi-smooth functions; see~e.g.~\cite{Gowda2004}.
\begin{lemma} \label{lem:implicit}
Let $T : \RR^m \times \RR^n \to \RR^m$ be semi-smooth in a neighborhood of $(x^*;y^*)$. Assume $T(x^*;y^*) = 0$ and that all matrices in $\partial_x T(x^*;y^*)$ are non-singular. Then, there exists an open neighborhood $N(y^*)$ of $y^*$ and a function $X(y^*) : N(y^*) \mapsto \RR^m$ which is locally Lipschitz and semi-smooth and such that $X(y^*) = x^*$ and $T(X(y);y) = 0$ for all $y \in N(y^*)$. Moreover, if $[A_x,A_y] \in \partial^B T(X(y);y)$, we have
\begin{align}
  -A_x^{-1}A_y \in \partial^B X(y) \subset \partial X(y).
\end{align}
\end{lemma}

By verifying the conditions of this auxiliary Lemma, we can show the following result, which is the key step for the proof of Theorem~\ref{thm:semismooth:b}.
\begin{lemma} \label{lem:J}
Let Assumption~\ref{ass:1}. Then for any $\hh \in \RR^d$, the system \eqref{eq:T} has a unique solution $(\jj,\lambda)$. Moreover, the mapping $\hh \mapsto \jj(\hh)$ is semi-smooth and  
\begin{align}
    \label{eq:slant:j}
    \mathbb{S}_\jj = \begin{cases}
    0, & \text{ if } \jj = \jj_p, \\
    \Big( \nabla_{\jj \jj}^2 U(\jj) + \frac{\chi}{|\jj - \jj_p|} \Big( I - \frac{(\jj - \jj_p) \,\otimes\, (\jj - \jj_p)}{|\jj - \jj_p|^2}\Big) \Big)^{-1}, & \text{ if } \jj \neq \jj_p,
    \end{cases}
\end{align}
is well-defined and an element in the generalized Jacobian $\partial \jj(\hh)$.
\end{lemma}
\begin{proof}
From the previous results, we immediately conclude that that for any $\hh \in \RR^d$, the system $T((\jj(\hh),\lambda(\hh);\hh) = 0$ has a unique solution. 
Further recall that $U(\cdot)$ is twice differentiable on its domain, and hence in a neighborhood of $\jj(\hh)$. Since the $\max$-function is semi-smooth~\cite{hintermuller2002}, we conclude that the mapping $T$ is semi-smooth in a neighborhood of the solution $(\jj(\hh),\lambda(\hh);\hh)$. 
Moreover, the partial generalized Jacobians $A_{(\jj,\lambda)} \in \partial_{(\jj,\lambda)} T((\jj,\lambda);\hh)$ and $A_{\hh} \in \partial_{\hh} T((\jj,\lambda);\hh)$ have the form
\begin{align}
    &A_{(\jj,\lambda)} = \left(\begin{matrix} 
        I + \lambda \nabla_{\jj \jj}^2 U & \nabla_{\jj} U - \hh \\ 
         (\nabla_{\jj}U - \hh)^\top (\nabla_{\jj \jj}^2 U) \partial_{1} \ncp & \partial_{2} \ncp
        \end{matrix} \right) \quad
        \text{ and } \notag \\
        &A_{\hh} = \left(\begin{matrix} 
        -\lambda\,I  \\ 
         -(\nabla_{\jj}U - \hh)^\top \partial_{1} \ncp
        \end{matrix} \right).
        \label{eq:mat:A}
\end{align}
Here $U$ is evaluated at $\jj$ and $\ncp$ at $x_1 = \frac{1}{2} |\nabla_{\jj}U(\jj) - \hh|^2 - \chi^2$ and $x_2 = \lambda$. The generalized Jacobian $\partial \ncp = (\partial_1 \ncp,\partial_2 \ncp)$ of the ncp-function is given by the set-valued mapping \cite{Bauschke2011}
\begin{align}
        \partial \ncp (x_1,x_2) = \left\{\begin{array}{ll}
    (0,-1), & \text{ if } x_1 + x_2 < 0,\\
    \{(\alpha,\alpha-1) | \alpha \in [0,1]\}, & \text{ if } x_1 + x_2 = 0,\\
    (1,0), & \text{ if } x_1 + x_2 > 0.
    \end{array}
    \right.
\end{align}
By considering several particular cases and using the equivalent definitions of $(\jj,\lambda)$ provided by Lemma~\ref{lem:system} and \ref{lem:T}, we now show that for any solution of $T((\jj,\lambda);\hh) = 0$, the partial Jacobian $A_{(\jj,\lambda)}$ in \eqref{eq:mat:A} is always non-singular. 
\paragraph{\textbf{Case 1: $x_1 + x_2 < 0$.}}
Since $(\jj,\lambda)$ is a solution of the system \eqref{eq:system:1}--\eqref{eq:system:4} this implies, in particular, that $x_2 = \lambda = 0$. Hence the partial Jacobian has the form
\begin{align*}
        A_{(\jj,\lambda)}^{\jj = \jj_p} = \left(\begin{matrix} 
        I   & \nabla_{\jj} U - \hh \\ 
         0  & -1
        \end{matrix} \right),
\end{align*}
which clearly defines a non-singular matrix.

\paragraph{\textbf{Case 2: $x_1 + x_2 = 0$.}} 
We can immediately deduce that $|\nabla_{\jj}U(\jj) - \hh| =  \chi$ and hence $\lambda = 0$. The partial Jacobians thus have the form 
\begin{align*}
        &A_{(\jj,\lambda)} = \left(\begin{matrix} 
        I  & \nabla_{\jj} U - \hh \\ 
         (\nabla_{\jj}U - \hh)^\top (\nabla_{\jj \jj}^2 U) \alpha & (\alpha-1)
        \end{matrix} \right)
    \end{align*}
    for $\alpha \in [0,1]$. 
    The Schur complement with respect to the second variable is given by 
    \begin{align*}
        \alpha (\nabla_{\jj}U - \hh)^\top (\nabla_{\jj \jj}^2 U) (\nabla_{\jj}U - \hh) + (1-\alpha).
    \end{align*}
    Using that $\alpha \in [0,1]$, $(\nabla_{\jj}U - \hh) \neq 0$, and $\nabla_{\jj \jj}^2 U$ is positive definite, we see that this matrix is positive definite, which implies that the matrix $A_{(\jj,\lambda)}$ is regular.

\paragraph{\textbf{Case 3: $x_1 + x_2 > 0$.}} 
From \eqref{eq:system:1}--\eqref{eq:system:4}, we know that $|\nabla_{\jj}U - \hh| = \chi$ and that $\lambda > 0$. Hence the partial Jacobians have the particular form 
\begin{align}
        &A_{(\jj,\lambda)} = \left(\begin{matrix} 
        I + \lambda \nabla_{\jj \jj}^2 U & \nabla_{\jj} U - \hh \\ 
         (\nabla_{\jj}U - \hh)^\top (\nabla_{\jj \jj}^2 U)  & 0
        \end{matrix} \right)
    \end{align}
    The same arguments like in Case~2 again yield the regularity of this matrix.

\smallskip 

\paragraph{\textbf{Semi-smoothness.}}
We have shown that for any choice $[A_{(\jj,\lambda)},A_\hh] \in \partial T((\jj,\lambda);\hh)$ in the generalized Jacobian, the matrix $A_{(\jj,\lambda)}$ is regular if $T((\jj,\lambda);\hh)=0$; hence Lemma~\ref{lem:implicit} is applicable and provides the semi-smoothness of $\jj(\hh)$ and $\lambda(\hh)$. 

\smallskip 

\paragraph{\textbf{Jacobians.}}
It remains to derive suitable candidates in the generalized Jacobian, for which we again consider two separate cases: 
If $\jj = \jj_p$, 
we know that $\lambda = 0$ and hence can choose the matrices $A_{\jj,\lambda}^{\jj = \jj_p} \in \partial_{(\jj,\lambda)} T$ and $A_{\hh}^{\jj = \jj_p} \in \partial_{\hh} T$ with
\begin{align}
    A_{(\jj,\lambda)}^{\jj = \jj_p} = \left(\begin{matrix} 
    I   & \nabla_{\jj} U - \hh \\ 
     0  & -1
    \end{matrix} \right) 
    \quad \text{and} \quad   
    A_{\hh}^{\jj = \jj_p} = \left(\begin{matrix} 
    0 \\ 
     0
    \end{matrix} \right).
\end{align}
Note that $[A_{\jj,\lambda}^{\jj = \jj_p},A_{\hh}^{\jj = \jj_p}]$ lies in the Bouligand subdifferential $\partial^B T((\jj,\lambda),\hh)$.
This choice immediately leads to $\mathbb{S}_{\jj} = 0$, as announced in the lemma. 
If $\jj \neq \jj_p$,
we know from Lemma~\ref{lem:system} that $\lambda = \frac{|\jj - \jj_p|}{\chi}$ and $\nabla_{\jj} U - \hh = \chi\frac{\jj - \jj_p}{|\jj - \jj_p|}$. We then choose
\begin{align*}
    A_{(\jj,\lambda)}^{\jj \neq \jj_p} = \left(\begin{matrix} 
    I + \frac{|\jj - \jj_p|}{\chi} (\nabla_{\jj \jj}^2 U) & -\chi\frac{\jj - \jj_p}{|\jj - \jj_p|} \\ 
    -\chi\frac{(\jj - \jj_p)^\top}{|\jj - \jj_p|}(\nabla_{\jj \jj}^2 U)  & 0
    \end{matrix} \right) 
    \quad \text{and} \quad   
    A_{\hh}^{\jj \neq \jj_p} = \left(\begin{matrix} 
    -\frac{|\jj - \jj_p|}{\chi} I \\ 
    \chi\frac{(\jj - \jj_p)^\top}{|\jj - \jj_p|}
    \end{matrix} \right),
\end{align*}
This choice again yields an element of the Bouligand subdifferential.
By solving $A_{(\jj,\lambda)}^{\jj \neq \jj_p} \binom{\mathbb{S}_{\jj}}{\mathbb{S}_{\lambda}} = -A_{\hh}^{\jj \neq \jj_p}$, we obtain the formula for $\mathbb{S}_{\jj}$ for this case.
\end{proof}

\subsection{Proof of Theorem~\ref{thm:semismooth:b}}
Noting that $\bb(\hh) = \mu_0 \hh + \jj(\hh)$, the result follows immediately from the additivity theorem for semi-smooth functions and Lemma~\ref{lem:J}. 

\section{A semi-smooth Newton method}
\label{sec:newton}

For the iterative solution of the nonlinear system \eqref{eq:scal_var_h}, we apply a semi-smooth Newton method with line search~\cite{Qi1993,Ulbrich2011}. 
Starting from $\psi_h^0 \in V_h$, we update 
\begin{align}
    \label{eq:update:ssn}
    \phi^{n+1}_h = \phi^n_h + \tau^n \delta \phi^n_h \qquad n \geq 0
\end{align}
with the search direction $\delta \phi^n_h \in V_h$ defined as the solution of the linear system
\begin{align}
    \label{eq:linear:system:ssn}
    (\mathbb{S}_{\bb}^n\,\nabla  \delta \phi^n_h, \nabla v_h)_h = (\nabla_{\hh} w^*(\hs - \nabla \psi_h) , \nabla v_h)_{h} \quad \forall v_h \in V_h.
\end{align}
On every element $T \in \Th$, the matrix $\mathbb{S}_{\bb}^n|_T$ is required which is chosen as the element in the generalized Jacobian $\partial_{\hh}\bb(\hs(x_T) - \nabla \psi^n_h(x_T))$ as described in \eqref{eq:slant:b}.
For determining the step size $\tau^n$, we use 
\emph{Armijo-backtracking}~\cite{Nocedal2006}, i.e.,   
we choose
\begin{align} 
\label{eq:armijo}
\tau^n = \max\{&\tau= \rho^m: 
\Phi(\tau) \le \Phi(0) + \sigma \tau \Phi'(0), \ m \ge 0\} \end{align}
with $\Phi(\tau)=(w^*(\hh_s - \nabla (\psi_h^n + \tau \delta \psi_h^n),1)_h$ and given $0<\sigma< 1/2$ and $\rho<1$. 

\begin{remark}
If the material law $\bb(\hh) = \nabla_\hh w^*(\hh)$ is sufficiently smooth, then the matrix $\mathbb{S}_{\bb}^n = \nabla_{\hh} \bb(\hh^n) = \nabla_{\hh\hh}^2 w^*(\hh^n) $ corresponds to the Hessian of the co-energy functional, and the method \eqref{eq:update:ssn}--\eqref{eq:linear:system:ssn} reduces to the damped Newton method considered in \cite{Engertsberger23}. Alternative choices for $\mathbb{S}_{\bb}^n$ that also work in the non-smooth setting have been considered in \cite{Egger2024quasi,Dlala2008b,Hantila2000}; see Section~\ref{sec:numerics} for more details.
\end{remark}

Based on our previous considerations, we obtain the following convergence result.
\begin{theorem} \label{thm:convergence}
Let Assumption \ref{ass:3} hold and $\psi_h \in V_h$ be the unique solution of \eqref{eq:scal_var_h}. 
Then for any $\psi_h \in V_h$, the iterates $\psi_h^n$ defined by \eqref{eq:update:ssn}--\eqref{eq:armijo} satisfy 
\begin{align*}
    \|\psi_h^n - \psi_h\|_h \leq C \, q^n \|\psi_h^0 - \psi_h\|_h, \qquad n \geq 1
\end{align*}
with constants $C \ge 0$ and $q < 1$ independent of $\Th$; hence the iterative method exhibits global $r$-linear and mesh-independent convergence. 
For $\psi_h^n$ sufficiently close to $\psi_h$, Armijo backtracking \eqref{eq:armijo} returns $\tau^n=1$ and 
$\|\psi_h^{n+1} - \psi_h\|_h = o(\|\psi_h^{n} - \psi_h\|_h)$ with $n \to \infty$; i.e., the method converges locally superlinearly.
\end{theorem}
\begin{proof}
From \eqref{eq:slant:b}, we see that either $\mathbb{S}_{\bb}^n = \mu_0 I$ or $\mathbb{S}_{\bb}^n = \mu_0 I + (\mathbb{P}^n)^{-1}$ with a symmetric and strictly positive definite matrix $\mathbb{P}^n$. 
Its inverse is at least positive semi-definite and uniformly bounded, which implies 
\begin{align}
    \label{eq:bounds:ssn}
    \mu_0 |\xi|^2 \leq \xi^\top \mathbb{S}_{\bb}^n \xi \leq (\mu_0 + 1/\sigma) |\xi|^2 \qquad n \geq 0.
\end{align}
From \cite[Thm. 2]{Egger2024quasi}, we can thus conclude global linear and mesh-independent convergence of the iteration. 
To prove that $\tau^n=1$ for sufficiently large $n$, we apply \cite[Thm.~3.3]{Facchinei1995}, which requires to verify the inequality
\begin{align*}
    -(\nabla_\hh w^*(\hs - \nabla \phi^n_h) , \nabla \delta \psi_h^n)_{h} \leq -p \, \|\delta \phi^n_h\|_h^2
\end{align*}
for all $n \geq n_0$ with a constant $p > 0$. 
This assertion follows from the definition of the semi-smooth Newton step in \eqref{eq:linear:system:ssn} and the bounds in \eqref{eq:bounds:ssn}. The superlinear convergence can then be deduced from \cite[Prop.~2.12]{Ulbrich2011}.
\end{proof}

\section{Generalizations}
\label{sec:extension}

Before turning to numerical tests, let us mention two generalizations of the proposed results that follow with similar arguments as the previous sections.

\subsection{Multiple pinning forces}
To archive better accuracy, the magnetic polarization can be split into partial magnetic polarizations $\jj = \sum\nolimits \jj_k$; see \cite{Lavet2013}. 
Each $\jj_k$ is then defined separately by a minimization problem of the form 
\begin{align}
    \jj_k = \argmin\nolimits_{\jj_k} \{ U_k(\jj_k) - \langle \hh, \jj_k \rangle + \chi_k |\jj_k - \jj_{k,p}| \}
\end{align}
with the partial magnetic polarization $\jj_{k,p}$ of the previous time-step, internal energy densities $U_k(\cdot)$, and pinning forces $\chi_k$, $k=1,\ldots,K$. The problem data are assumed to uniformly satisfy Assumption \ref{ass:3}. 
The corresponding hysteresis operator then has the form $\bb(\hh) = \mu_0 \hh + \sum\nolimits_k \jj_k$. It is not difficult to see that this material law can again be restated as $\bb(\hh) = \nabla_{\hh} w^*(\hh)$ with magnetic co-energy density \cite{egger2024inverse}
\begin{align} \label{eq:wsk}
w^*(\hh) = \frac{\mu_0}{2} |\hh|^2 - \sum\nolimits_{k} \Big( \min_{\jj_k} \{ U_k(\jj_k) - \langle \jj_k, \hh \rangle + \chi_k |\jj_k - \jj_{k,p}| \} \Big).
\end{align}
The properties of the co-energy density $w^*(\hh)$ and the hysteresis operator $\bb(\hh)$ stated in the previous sections carry over almost verbatim to this setting. 
In particular, an element in the generalized Jacobian $\mathbb{S}_{\bb} \in \partial_{\hh} \bb(\hh)$ is given by
\begin{align}
    \mathbb{S}_{\bb} = \mu_0 I + \sum\nolimits_k \mathbb{S}_{\jj_k},
\end{align}
where the generalized Jacobians $\mathbb{S}_{\jj_k}$ are defined in the same way as in Lemma~\ref{lem:J}. 
Using $\bb(\hh) = \mu_0 \hh + \sum\nolimits_k \jj_k$, the convergence statements about the semi-smooth Newton method then carry over immediately to the case of multiple pinning forces.

\subsection{Inverse hysteresis operator}
By convex duality, see e.g. \cite[Ch.~5]{Beck2017}, the inverse of the strongly monotone and Lipschitz continuous function $\bb(\hh) = \nabla_\hh w^*(\hh)$ can be represented in the form $\hh(\bb) = \nabla_\bb w(\bb)$, where 
\begin{align}
w(\bb) = \sup\nolimits_\hh \{\langle \hh,\bb\rangle - w^*(\hh) \}
\end{align}
is the convex conjugate function, i.e., the \emph{magnetic energy density}. 
In the context of magnetic hysteresis, $\hh(\bb)$ is called the \emph{inverse hysteresis operator}.
By the conjugate correspondence principle \cite[Thm.~5.26]{Beck2017}, Lipschitz continuity and  strong convexity of $\bb(\hh)$ translate to strong convexity and Lipschitz continuity of $\hh(\bb)$. 
The inverse function theorem for semi-smooth functions~\cite{Gowda2004,Pang2003} further implies that semi-smoothness of $\bb(\hh)$ translates to that of $\hh(\bb)$, and that a representative in generalized Jacobian $\partial_{\bb} \hh(\bb)$ of the inverse hysteresis operator is given by $\mathbb{S}_{\bb} = (\mathbb{S}_{\bb})^{-1}$ with 
$\mathbb{S}_{\bb}$ as defined in \eqref{eq:slant:b}. 
The results derived in this paper can thus be transferred to the vector potential formulation of \eqref{eq:1}--\eqref{eq:3}, which reads 
\begin{alignat}{2}
\curl(\nabla_\bb w(\curl \aa)) &= \js \qquad &&\text{in } \Omega, \\
\nn \times \aa &= 0 \qquad && \text{on } \partial\Omega,
\end{alignat}
and corresponding discretizations thereof; see \cite{Egger2024ho,heise94} for the anhysteretic case. Numerical results involving the inverse hysteresis operator can be found in \cite{egger2024inverse}.

\section{Numerical tests}
\label{sec:numerics}

To illustrate the theoretical results derived in this paper and to demonstrate the efficiency of the proposed semi-smooth Newton method, we simulate the T-joint of a three-limb transformer. the setup of our test problem is taken from~\cite{Gyselinck2004}. 

\subsection{Model problem}
The specific setup of the problem allows to perform the computations on a two-dimensional cross-section $\Omega \subset \RR^2$ which is depicted in Figure \ref{fig:geometry:tjoint}. 
\begin{figure}[!ht]
    \centering
    \includegraphics[trim = {1cm 1cm 0cm 0.5cm},clip,width=0.6\linewidth]{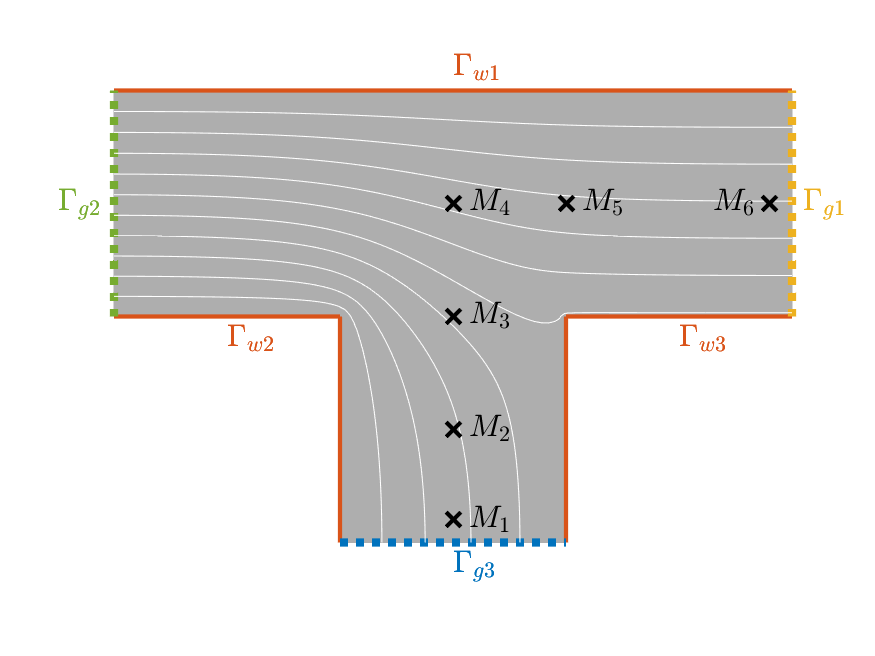}
    \caption{2D cross-section of the T-joint of a transformer (limb width: $1$m) taken from \cite{Gyselinck2004,Jacques2015}. On the flux walls $\Gamma_w$ (red) the boundary condition $\bb \cdot \mathbf{n} = 0$ is imposed, while on the flux gates $\Gamma_g$ (green, yellow, blue) we prescribe $\hh \times \mathbf{n} = 0$. In addition, a time-varying flux $\phi$ is imposed at these gates. 
    Corresponding $\bb-\hh$ values at different time-steps are extracted at the measurement points $M_1$--$M_6$.}
    \label{fig:geometry:tjoint}
\end{figure}
On the flux walls $\Gamma_w$ the boundary condition $\bb \cdot \mathbf{n} = 0$ is imposed, while on the flux gates $\Gamma_g$, we require $\hh \times \mathbf{n} = 0$. 
In addition, the total flux $\Phi_i = \int_{\Gamma_{gi}} \bb \cdot \nn \, ds$ is imposed. The condition
$\Phi_3 = -(\Phi_1 + \Phi_2)$ is required to stay compatible with the magnetic Gau\ss' law. 
\begin{table}[h!]
\centering
\small
\begin{tabular}{@{}cll@{}}
\toprule 
\textbf{parameter} & \textbf{values} & \textbf{description}\\
\midrule
$A_s$ & 65 & field strength parameter \\
$J_{s,k}$  & 0.11, 0.3, 0.44, 0.33, 0.04 & partial reference polarization\\
$\chi_k$  & 0, 10, 20, 40, 60 & pinning force associated with $J_{s,k}$\\
\bottomrule
\end{tabular}
\caption{Material data for 5-cell model from \cite{Lavet2013}.}
\label{tab:material:lavet}
\end{table}
A magnetic hysteresis model with $5$ pinning forces and  internal energy densities $U_k(\jj_k) = - \frac{A J_{s,k}}{\pi} \log \left(\cos \left(\frac{\pi}{2} \frac{|\jj|}{J_{s,k}}\right)\right)$ 
is used in our simulations. The corresponding parameters are summarized in Table~\ref{tab:material:lavet}. 
Let us note that the internal energy densities $U_k$ satisfy all the conditions required in our assumptions.

\subsection{Discretized scalar potential formulation}

The following variational problem, describing the finite element approximation for the weak formulation of the problem under consideration, can be derived using standard arguments.
\begin{problem} \label{prob:test}
Find $\psi_h \in V_h$ such that 
    \begin{align}
        \label{eq:numerics:scal}
        (\nabla_\hh w^*(-\nabla \psi_h), \nabla v_h)_h = (\Phi,w_h)_{\Gamma_g} \qquad \forall v_h \in V_h
    \end{align}
with $V_h:= \{v_h \in P_1(\Th) \cap H^1(\Omega): v_h|_{\Gamma_i}=c_i, \, c_1=0, \ c_2,c_3 \in \RR \}$.
\end{problem}
Here $P_1(\Th) \cap H^1(\Omega)$ is the space of continuous piecewise linear finite element functions on a triangular mesh $\Th$ of $\Omega$, and $(a,b)_\Gamma = \int_\Gamma a b$ abbreviates the integral over the boundary. 
All results of this work are also valid for this particular problem.

\subsection{Iterative solvers}
In our numerical tests, we compare different iterative methods for the solution of Problem~\ref{prob:test}, which all have the common form \eqref{eq:update:ssn}--\eqref{eq:armijo}. The only difference consists in the definition of the matrix $\mathbb{S}_{\bb}^n$ for determining the update direction $\delta \psi_h^n$. 
The following choices are considered:
\begin{itemize}
\item the semi-smooth Newton method (SSN) presented in Section~\ref{sec:newton};
\item the local Quasi-Newton Method (LQN) proposed in \cite{Egger2024quasi} with BFGS updates;
\item the local coefficient method (LCM) described in \cite{Dlala2008b}; 
\item the global coefficient method (GCM) presented in \cite{Hantila2000}.
\end{itemize}
While the first two methods update the matrices $\mathbb{S}_{\bb}^n$ in every iteration step, the other methods work with fixed matrices $\mathbb{S}_{\bb}$ for all iterations of a single load step. For (LCM) the approximation is formally defined by 
\begin{align*}
\mathbb{S}_{\bb} = \frac{1}{2} \Big( \frac{d \bb_x^0}{d \hh_x^0} + \frac{d \bb_y^0}{d \hh_y^0} \Big)
\end{align*}
which is evaluated at every quadrature point $x_T$. 
The derivatives are approximated by difference quotients in the first iteration step, always checking if the values are above/below the parameters $\mu_0$ and $\nu_0$ in free space respectively. 
For the (CGM), one simply chooses $\mathbb{S}_{\bb}= \mu_0 \mu_r I$ with the relative permeability $\mu_r = 1000$ determined by parameter tuning. Here the matrix stays the same even over a complete load cycle. 

\subsection{Implementation details}
All methods are considered as special cases of the iteration \eqref{eq:update:ssn}--\eqref{eq:armijo} with the only difference being the choice of the matrix $\mathbb{S}_{\bb}^n$ in every iterations step. 
The solution of the linear systems \eqref{eq:linear:system:ssn} is done by sparse direct solvers. For the (LCM) and (GCM) methods, factorizations are only performed once per load step.
For the Armijo linesearch \eqref{eq:armijo} we use $\rho = 0.5$, $\sigma = 0.1$. The methods are terminated when $|W_h^*(\psi_h^{n+1}) - W_h^*(\psi_h^{n})| \leq 10^{-8} \, W_h^*(\psi_h^{0})$ for the first time; here $W_h^*(\psi_h)=(w^*(\hh_s - \nabla \psi_h,1)_h$ denotes the total co-energy of the system computed on the current grid.
All computations were carried out in \textsc{Matlab} on a laptop equipped with a 13th Gen Intel Core i5-1335U CPU with clock rate 1.30 GHz.

\subsection{Simulation of a single load step}
In a first experiment, we illustrate the global convergence of the proposed method by the simulation of only one large load step. Here the previous partial polarizations are set to $\jj_{k,p} = 0$ and the input flux to $\Phi_1 = -0.5$ resp. $\Phi_2 = 1$. The corresponding results are depicted in Table \ref{tab:exp1:scal}.

\begin{table}[ht!]
\centering \small
\begin{tabular}{@{}r rr rr rr rr@{}}
\toprule
\textbf{dof} 
& \multicolumn{2}{c}{\textbf{SSN}} 
& \multicolumn{2}{c}{\textbf{LQN}} 
& \multicolumn{2}{c}{\textbf{LCM}} 
& \multicolumn{2}{c}{\textbf{GCM}} \\
\midrule
570    & \hspace*{1em} 6 & 0.062s   & \hspace*{1em} 14 & 0.147s  & \hspace*{1em} 28 & 0.240s   & \hspace*{1em} 43 & 0.440s \\
2,170  & 6 & 0.146s   & 15 & 0.527s  & 39 & 1.025s   & 70 & 2.310s \\
8,563  & 6 & 0.576s   & 15 & 2.180s  & 63 & 6.680s   & 120 & 16.10s \\
33,421 & 6 & 2.710s   & 21 & 14.40s  & 97 & 47.10s   & 204 & 130.00s \\
\bottomrule
\end{tabular}
\caption{Iteration numbers and computation times for the four different methods across mesh refinement levels for simulation of a single load step.}
\label{tab:exp1:scal}
\end{table}

As predicted by our theoretical results, the semi-smooth Newton method (SSN) converges globally with mesh independent iteration numbers. The local quasi Newton method (LQN) requires around 3 times as many iterations and computation time. The local and global coefficient method (LCM) and (GCM) seem to suffer from the bad initialization and hence require significantly more iterations and time. We also observe a mild increase in the iterations numbers with the system dimension, which we interpret as an improved performance on coarse meshes, i.e.,  in the pre-asymptotic regime. In principle, also for these methods the global linear mesh-independent convergence is guaranteed by \cite[Thm.~2]{Egger2024quasi}.

\subsection{Simulation of a load cycle}
In the second experiment, we consider simulation of a typical load cycle. We start in the completely demagnetized state with partial magnetic polarization $\jj_{k,p} = 0$ and impose two periods of the sinusoidal flux input $\Phi_1(t) = \cos(2\pi t + 2\pi/3)$ resp. $\Phi_2(t) = \cos(2\pi t)$, which in the first quarter of the initial period are multiplied by the factor $(1 - \cos(4 \pi t))/2$; see \cite{Gyselinck2004} for a similar setup. 
Let us note that due to the rate-independence of the model the scaling of time is not of relevance.
The time interval $[0,2]$ is decomposed into 100 equidistant time steps and in every time step $t^n$ the magnetic field problem \eqref{eq:numerics:scal} is solved with $\jj_{p,k} = \jj_k(t^{n-1})$, where $\jj_k(t^{n-1})$ is the partial magnetic polarization of the pervious time step.
Our observations for these simulations are summarized in Table~\ref{tab:average:scal}.

\begin{table}[ht!]
\centering \small
\begin{tabular}{@{}r rr rr rr rr@{}}
\toprule
\textbf{dof} 
& \multicolumn{2}{c}{\textbf{SSN}} 
& \multicolumn{2}{c}{\textbf{LQN}} 
& \multicolumn{2}{c}{\textbf{LCM}} 
& \multicolumn{2}{c}{\textbf{GCM}} \\
\midrule
570 & \hspace*{1em} 4.30 & 3.51s & \hspace*{1em} 5.60 & 4.03s  & \hspace*{1em} 16.6 & 15.1s & \hspace*{1em}29.3 & 24.7s   \\
2170 &  4.35 & 11.3s &  5.69 & 13.7s  &  17.6 & 54.9s & 34.2 & 93.2s  \\
8563 &  4.42 & 50.7s &  5.79 & 60.8s  & 19.2 & 249.8s & 44.5 & 535.2s    \\
33421 &  4.52 & 244.2s &  6.01 & 301.1s  & 26.4 & 1608.5s & 57.1 & 3115.5s    \\
\bottomrule
\end{tabular}
\caption{Average iteration numbers per load step and total computation times for different methods and refinement levels for a load cycle with 100 time steps}
\label{tab:average:scal}
\end{table}

Since the iterations in every time step after the first are started with a good initial guess, we expect a general improvement of average iteration numbers. 
Again the semi-smooth Newton method (SSN) performs best with the lowest number of iteration and computation time. 
The local quasi Newton method (LQN) requires about 1.3-1.4 times more iterations and computation here. 
The local coefficient method (LCM) needs about 4-6 times as many iterations and while the numbers for the global coefficient method (GCM) are increased by another factor of about 2. 
Iteration numbers of 30-100 were also observed for modifications of the LCM and GCM method~\cite{Yue2025,Zhou2017}, so the two Newton-type methods (SSN) and (LQN) seem to constantly provide a superior performance. 
%


For illustration of the material behavior, we display in Figure~\ref{fig:hysteresis:HB} 
the hysteresis curves for the time-interval $[1,2]$ at the reference points $M_i$, $i=1,\ldots,6$ computed by the damped semi-smooth Newton method on a mesh with $33\;421$ vertices.
\begin{figure}[ht!]
    \centering
    \includegraphics[trim = {0.5cm 0cm 1cm 0.5cm},clip,width=0.45\linewidth]{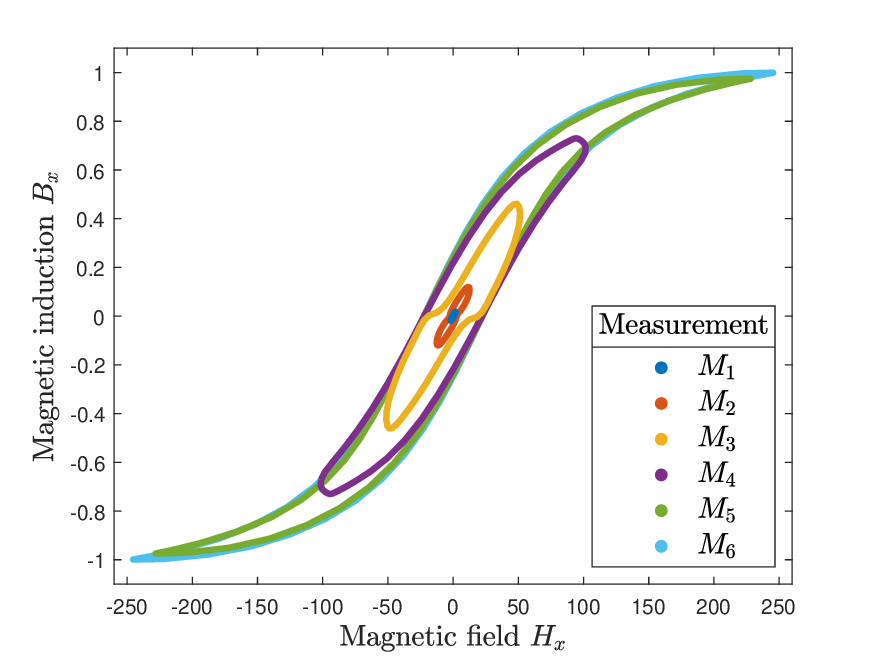}
    \hfill
    \includegraphics[trim = {0.5cm 0cm 1cm 0.5cm},clip,width=0.45\linewidth]{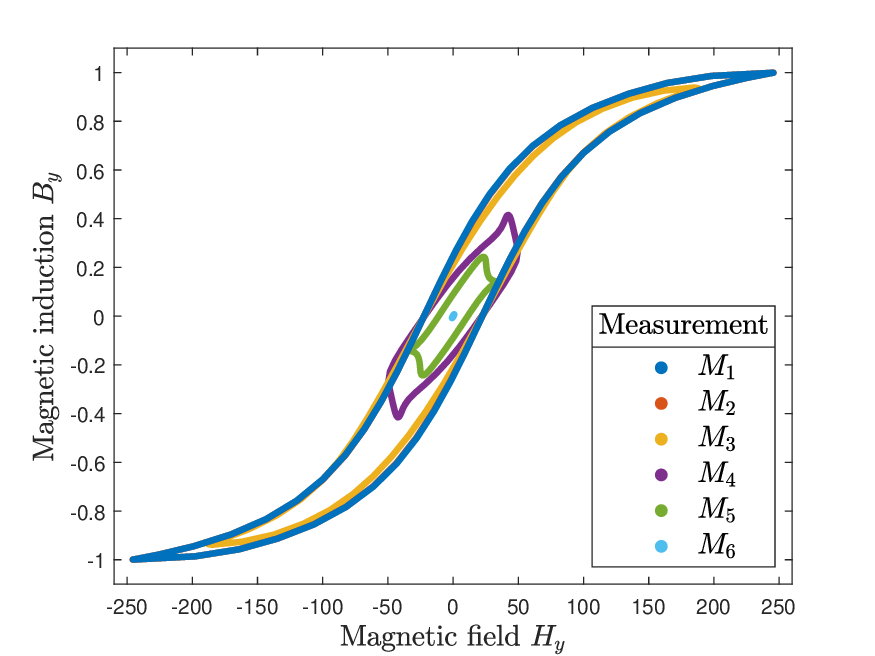}
    \caption{Hysteresis curves of the $x$- and $y$-component of the magnetic field $\hh$ and magnetic induction $\bb$ at the measurement points $M_i$ depicted in Figure~\ref{fig:geometry:tjoint}.}
    \label{fig:hysteresis:HB}
\end{figure}

\section{Discussion}
In this work we studied the constitutive relation $\bb=\bb(\hh;\jj_p)$ resulting from the energy-based hysteresis model of \cite{Lavet2013}. We established strong monotonicity and Lipschitz continuity, which allowed us to prove well-posedness of the scalar potential formulation of corresponding nonlinear magnetic field problems. 
In addition, we proved semi-smoothness of the hysteresis operator and investigated the local and global convergence properties of a corresponding semi-smooth Newton method. 
The feasibility and performance of the method was demonstrated by numerical tests for a typical test problem. 
As indicated in Section~\ref{sec:extension}, a similar analysis could be performed for the inverse hysteresis operator $\hh=\hh(\bb;\jj_p)$ and the vector potential formulation, which would also allow to include the effect of eddy currents very naturally. Iterative solvers for related problems have been studied, e.g., in \cite{Saitz1999,Yuan2005}. Detailed investigations in these directions are left for future research.

\bigskip

\begingroup
\small 
\subsection*{Acknowledgements}
This work was supported by the joint DFG/FWF Collaborative Research Centre CREATOR (DFG: Project-ID 492661287/TRR 361; FWF: 10.55776/F90).
\endgroup

\bigskip 


\end{document}